\documentclass[12pt]{amsart}

\usepackage{amssymb}
\usepackage{amscd}
\usepackage{color}

\theoremstyle{plain}
  \newtheorem{theorem}{Theorem}[section]
  \newtheorem{proposition}[theorem]{Proposition}
  \newtheorem{lemma}[theorem]{Lemma}
  \newtheorem{corollary}[theorem]{Corollary}
  
\theoremstyle{definition}
  
  \newtheorem{example}[theorem]{Example}

\numberwithin{equation}{section}

\newcommand{\integers}{\mathbb{Z}}
\newcommand{\field}{\mathbb{K}}
\newcommand{\SCMk}{${\mathrm{SCM}}/\field$}
\newcommand{\SHCM}{${\mathrm{SHCM}}$}
\newcommand{\CMk}{${\mathrm{CM}}/\field$}
\newcommand{\HCM}{${\mathrm{HCM}}$}
\newcommand{\link}{\mathrm{lk}}
\newcommand{\Star}{\mathrm{st}}
\newcommand{\cone}{\mathrm{Cone}}

\newcommand{\fdel}{\mathrm{fdel}}
\newcommand{\rank}{\mathrm{rank}}
\newcommand{\etal}{et al\mbox{.}}

\begin{document}
  
  \title[Complexes of Injective Words]%
  {Complexes of Injective Words and Their Commutation Classes}
  
  \author{Jakob Jonsson}
  
  \address{
	Department of Mathematics \\
	KTH \\
	10044 Stockholm \\
	Sweden}

  \email{jakobj@math.kth.se}
  
  \author{Volkmar Welker}
  \address{Fachbereich Mathematik und Informatik\\
    Philipps-Universit\"at Marburg\\
    35032 Marburg, Germany}
  \email{welker@mathematik.uni-marburg.de}
  
  \thanks{First author supported by Graduiertenkolleg `Combinatorics, Geometry, Computation`, DFG-GRK 588/2.  
    Both authors were supported by EU
    Research Training Network 
    ``Algebraic Combinatorics in Europe'', grant HPRN-CT-2001-00272.}
  
  
  \subjclass{ }
  
  \begin{abstract}
    Let $S$ be a finite alphabet.
    An injective word over $S$ is a word over $S$
    such that each letter in $S$ appears at most once in the word.
    We study Boolean cell complexes of injective words over 
    $S$ and their commutation classes.
    This generalizes work by Farmer and by Bj\"orner and Wachs on
    the complex of all injective words. 
    Specifically, for an abstract simplicial complex $\Delta$, we
    consider the Boolean cell complex $\Gamma(\Delta)$ whose cells are
    indexed by all injective words over the sets forming the faces of 
    $\Delta$. 
    \begin{itemize}
    \item[$\triangleright$]
      For a partial order $P = (S, \leq_P)$ on $S$, we study the 
      Boolean cell complex $\Gamma(\Delta,P)$ of all words from
      $\Gamma(\Delta)$ whose sequence of letters 
      comes from a linear extension of $P$. 
    \item[$\triangleright$]
      For a graph $G = (S,E)$ on vertex set $S$, we study the
      Boolean cell complex $\Gamma/G(\Delta)$ whose cells are
      indexed by commutation classes $[w]$ of words from 
      $\Gamma(\Delta)$. More precisely, $[w]$ consists of all
      words that can be obtained from $w$ by successively 
      applying commutations
      of neighboring letters not 
      joined by an edge of $G$.
    \end{itemize}
    
    \noindent Our main results are as follows:
    \begin{itemize}
    \item[$\triangleright$]
      If $\Delta$ is shellable then so are $\Gamma(\Delta,P)$ and 
      $\Gamma/G(\Delta)$.
    \item[$\triangleright$] If $\Delta$ is Cohen-Macaulay (resp. 
      sequentially Cohen-Macaulay) then so are 
      $\Gamma(\Delta,P)$ and $\Gamma/G(\Delta)$.
    \item[$\triangleright$] The complex $\Gamma(\Delta)$ is
      partitionable.
    \end{itemize}
  \end{abstract} 

  \maketitle
  
  \section{Introduction}

  A word $\omega$ over a finite alphabet $S$ is called {\em injective} if 
  no letter appears more than once; that is $\omega  = \omega_1 \cdots
  \omega_r$ for  
  some $\omega_1, \ldots, \omega_r \in S$
  and $\omega_i \neq \omega_j$ for $1 \leq i < j \leq r$. 
  For $n + 1 = \# S$ we denote by $\Gamma_n$ the set of all injective
  words on $S$. A word  $\omega  = \omega_1 \cdots \omega_r$ 
  with $r$ letters is said to be of length $r$.
  A subword of a word $\omega_1 \cdots \omega_r$ is a word $\omega_{j_1}
  \cdots \omega_{j_s}$ such that $1 \leq j_1 < \cdots < j_s \leq r$.
  Clearly, a subword of an injective word is injective. 
  We order $\Gamma_n$ by saying that $\rho_1 \cdots \rho_s
  \preceq   \omega_1 \cdots \omega_r$ if and only if $\rho_1 \cdots \rho_s$
  is a subword of $\omega_1 \cdots \omega_r$. We write 
  $c(w)$ for the content $\{ \omega_1, \ldots, \omega_r \}$ of the
  word $w = \omega_1 \cdots \omega_r$. Then for any $A
  \subseteq c(w)$ there is a unique subword $v \preceq w$ of $w$ with
  $A = c(v)$.  
  This implies the well known fact (see \cite{Farmer}) that
  $\Gamma_n$ together with the partial order $\preceq$ is the face poset of
  a Boolean cell complex. Recall that a Boolean cell complex is a 
  regular CW-complex for which the poset of faces of each cell is
  a Boolean lattice. Clearly, simplicial complexes are special cases
  of Boolean cell complexes. From now on we will identify the poset $\Gamma_n$
  with the Boolean cell complex with face poset $\Gamma_n$. 
  In particular, we also identify the injective words of
  length $d+1$ with $d$-cells. Thus the faces of a given $d$-cell
  $w$ are the cells corresponding to all subwords of $w$.

  The complex $\Gamma_n$ is a well-studied object. Farmer
  \cite{Farmer} demonstrated that $\Gamma_n$ is homotopy equivalent to
  a wedge of spheres of top dimension. Bj\"orner and Wachs \cite{BjornerWachs}
  proved the stronger result that $\Gamma_n$ is shellable. See Reiner
  and Webb \cite{ReinerWebb} and Hanlon and Hersh \cite{HanlonHersh} for further
  refinements. Our generalizations are partly motivated by 
  specific examples of complexes of injective words that are
  used in algebraic $K$-theory; see e.g. \cite{Gerdes,MirzaiivanderKallen,Kerz,Knudson,Suslin,vanderKallen}.
  We will make this connection a bit more precise after Example \ref{matroidexample}. 

  All our simplicial complexes and Boolean cell
  complexes are assumed to be finite. 

  In this paper, we generalize $\Gamma_n$ in three directions:
  \begin{itemize}
  \item
    Given a simplicial subcomplex $\Delta$ on ground set $S$, we define a
    subcomplex $\Gamma(\Delta)$ of $\Gamma_n$ by restricting to 
    injective words $w \in \Gamma_n$ such that 
    the content $c(w)$ is a face of $\Delta$.
  \item
    Given a partially ordered set $P = (S,{\le_P})$ on the alphabet
    $S$, we define a subcomplex of $\Gamma_n$ by restricting to words
    $\omega_1 \cdots \omega_r$ such that $i<j$ whenever $\omega_i <_P
    \omega_j$. For a simplicial 
    complex $\Delta$ on $S$ we write $\Gamma(\Delta,P)$ for the
    set of all words $w \in \Gamma(\Delta)$ satisfying this restriction.
    In particular, $\Gamma(\Delta,P) \cong \Delta$ if $P$ is a total
    order and $\Gamma(\Delta,P) = \Gamma(\Delta)$ 
    if $P$ is an antichain.
  \item
    Given a graph $G = (S,E)$ on the alphabet $S$, we define the 
    equivalence class $[w]$ of an injective word $w \in \Gamma_n$ as the
    set of all words $v$ that can be obtained from $w$ by
    applying a sequence of commutations $ss' \rightarrow s's$ such
    that $\{s,s'\}$ is not an edge in $E$.
    For a simplicial complex $\Delta$ over $S$ we write
    $\Gamma/G(\Delta)$ for the set of equivalence classes $[w]$ of injective 
    words $w$ with content $c(w)$ in $\Delta$.  
    We order $\Gamma/G(\Delta)$ by saying $[v] \preceq [w]$ if there
    are representatives $v' \in [v]$ and $w' \in [w]$ such that
    $v' \preceq w'$. In particular, if $E = \emptyset$ then $\Gamma/G(\Delta) 
    \cong \Gamma(\Delta)$.
  \end{itemize}

  It is easy to see that $\Gamma(\Delta)$ and $\Gamma(\Delta,P)$ are
  lower order ideals in $\Gamma_n$ and therefore can be seen as
  subcomplexes of $\Gamma_n$. This in turn implies that we can also regard
  them as Boolean cell complexes. Slightly more care is needed to
  recognize $\Gamma/G(\Delta)$ as a Boolean cell complex. 

  \begin{lemma}
    $\Gamma/G(\Delta)$ is a Boolean cell complex.
  \end{lemma}
  \begin{proof}
    Clearly, if two words are in the same equivalence class, then their 
    contents must coincide.
    Also $[v] \preceq [w]$ implies that the content of $v$ is a subset of
    the content of $w$. These facts show that for a word $w$ of length
    $r$ there is a surjective poset map from $\{ [v]~|~[v] \preceq [w] \}$ 
    to the Boolean lattice of subsets of an $r$-element set. In order to
    show that it is an isomorphism we need to see that if $v_1$ and $v_2$ are
    words with the same content and $[v_1], [v_2] \preceq [w]$
    then $[v_1] = [v_2]$. We may assume that $v_1 \preceq w$ and $v_2
    \preceq w'$ 
    for some $w' \in [w]$. Since $w$ and $w'$ are equivalent, there is a
    sequence of commutations that leads from $w$ to $w'$. The
    commutations that involve only letters from $c(v_1)$ can then be used
    to move from $v_1$ to $v_2$. In particular, $[v_1] = [v_2]$. 
  \end{proof}

  The following three theorems are our main results. Their proofs are
  provided in the subsequent sections.
  The concepts from topological combinatorics used to formulate the 
  theorems are introduced in the corresponding section. For further
  reference we refer to the survey article by Bj\"orner \cite{Bjorner95} and
  for particular information about sequential Cohen-Macaulayness (CM)
  to Bj\"orner \etal\ \cite{BjornerWachsWelker07}.

  \begin{theorem}
    Let $\Delta$ be a shellable simplicial complex on the vertex set
    $S$.
    \begin{itemize}
    \item[(i)]
      Let $P = (S,\leq_P)$ be a partial order on $S$. Then the Boolean
      cell complex $\Gamma(\Delta,P)$ is shellable.
    \item[(ii)]
      Let $G = (S,E)$ be a simple graph on $S$. Then the Boolean
      cell complex $\Gamma/G(\Delta)$ is shellable.
    \end{itemize}
    \label{shell-thm}
  \end{theorem}

  Using the preceding theorem for $\Delta$ being a simplex and poset fiber
  theorems from \cite{BjornerWachsWelker05}, we derive our second main result.
  \begin{theorem}
    Let $\Delta$ be a 
    sequentially homotopy CM 
    (resp\mbox{.} sequentially CM over $\field$)
    simplicial complex on the vertex set $S$.  
    \begin{itemize}
    \item[(i)] If $P = (S, \leq_P)$ is a partial order on $S$, then
      the Boolean cell complex $\Gamma(\Delta,P)$ is 
      sequentially homotopy CM 
      (resp\mbox{.} sequentially CM over $\field$).
      In particular, if $\Delta$ is 
      homotopy CM 
      (resp\mbox{.} CM over $\field$),
      then so is $\Gamma(\Delta,P)$.
    \item[(ii)] If $G = (S,E)$ is a graph on vertex set $S$, 
      then the Boolean cell complex $\Gamma/G(\Delta)$ is 
      sequentially homotopy CM 
      (resp\mbox{.} sequentially CM over $\field$).
      In particular, if $\Delta$ is 
      homotopy CM 
      (resp\mbox{.} CM over $\field$),
      then so is $\Gamma/G(\Delta)$. 
    \end{itemize}
    \label{cm-thm}
  \end{theorem}
 
  Our third result exhibits a general property of the complexes
  $\Gamma(\Delta,P)$ in the case that $P$ is the antichain.

  \begin{theorem}
    Let $\Delta$ be a simplicial complex on vertex set $S$. 
    Then the complex $\Gamma(\Delta)$ of injective words derived from 
    $\Delta$ is partitionable.
    \label{partition-thm}
  \end{theorem}

  \section{Auxiliary lemmas}

  In this section we list some lemmas that give more insight into the
  structure of the complexes $\Gamma(\Delta,P)$ and $\Gamma/G(\Delta)$
  and also serve as
  ingredients for the proofs in later sections.

  For a partial order $P = (S,\leq_P)$
  and a set $B \subseteq S$, let $P|_B$ be the induced partial order
  on $B$. For a linear extension $w = a_1a_2\cdots a_r$ of $P|_{\{a_1,
  a_2, \ldots, a_r\}}$, $P +
  w$ denotes the partial order obtained 
  from $P$ by adding the relations $a_i<a_{i+1}$ for $1 \le i \le r-1$
  and taking the transitive closure of the resulting set of
  relations. For example, if $a, b$ are incomparable in $P$ then $P + ab$ 
  is the partial order obtained from
  $P$ by adding the relation $c<d$ for all pairs $(c,d)$ satisfying $c
  \le_P a$ and $b \le_P d$.

  For a Boolean cell complex $\Gamma$ and a face $w \in
  \Gamma$, let $\fdel_\Gamma(w)$ denote the complex obtained
  by removing the face $w$ and all faces containing $w$. 
  Let $\Star_\Gamma(w)$ be the complex consisting of all 
  faces $w'$ of $\Gamma$ such that some face of $\Gamma$ contains
  both $w$ and $w'$. 
  We call $\fdel_\Gamma(w)$ the \emph{deletion} of $w$ in $\Gamma$ and
  $\Star_\Gamma(w)$ the \emph{star} of $w$ in $\Gamma$. 

  If $\Delta$ and $\Gamma$ are simplicial complexes on
  disjoint ground sets, then the \emph{join} $\Delta * \Gamma$ is the
  simplicial complex 
  $$
  \Delta * \Gamma := 
  \{ \sigma \cup \tau:\sigma \in \Delta, \tau \in \Gamma \}.
  $$
  Still restricting to simplicial complexes, we define
  the \emph{link} of $\tau$ in 
  $\Gamma$ to be the simplicial complex 
  $$
  \link_\Gamma(\tau) := \{\sigma:\sigma \cup \tau  
  \in \Gamma, \sigma \cap \tau = \emptyset\}.
  $$
  Clearly,
  for a simplicial complex $\Gamma$ we have that
  $\Star_\Gamma(\tau) = 2^\tau * \link_\Gamma(\tau)$.

  \begin{lemma}
    Let $P = (S,\leq_P)$ be a partial order on $S$ and let 
    $\sigma$ and $\tau$ be subsets of $S$. If $w$ is a linear
    extension of $P|_\sigma$, then there exists a linear extension $w'$
    of $P|_{\sigma \cup \tau}$ containing $w$ as a subword.
    \label{extend-lem}
  \end{lemma}
  \begin{proof}
  By a simple induction argument, it suffices to consider the case
  $\tau = \{x\}$, where $x \in S \setminus \sigma$.
  If $x$ is maximal in $P|_{\tau +x}$, then we may choose $x$ to be the
  maximal element in $w'$. Otherwise, let $y$ be the leftmost element
  in $w$ such that $x <_P y$. Let $w'$ be the word
  obtained from $w$ by inserting $x$ just before $y$. 
  Then $w'$ is a linear extension of $P|_{\sigma+x}$. 
  Namely, if $z <_P x$, then $z <_P y$, which implies that $z$ appears
  before $y$ and hence before $x$ in $w'$.
  \end{proof}

  \begin{lemma}
    Let $\Delta$ be a simplicial complex on the vertex set $S$ and let
    $P = (S, \leq_P)$ be a partial order on $S$. If $w \in
    \Gamma(\Delta,P)$, then 
    $$
    \Gamma(\Star_{\Delta}(c(w)),P+w) = 
    \Star_{\Gamma(\Delta,P)}(w)
    $$
    \label{star-lem}
  \end{lemma}
  \begin{proof} A cell $w'$ belongs to the
    left-hand side if and only if $c(w) \cup c(w') \in \Delta$ and
  $w'$ is a linear extension of $(P+w)|_{c(w')}$. By
  Lemma~\ref{extend-lem}, there is then a face 
  $w''$ of $\Gamma(\Delta,P+w)$ containing $w'$ such that
  $c(w'') = c(w) \cup c(w')$. This implies that
  $w' \in \Star_{\Gamma(\Delta,P+w)}(w)$. 
  As a consequence, 
  $$
  \Gamma(\Star_{\Delta}(c(w)),P+w) \subseteq
  \Star_{\Gamma(\Delta,P+w)}(w) =
  \Star_{\Gamma(\Delta,P)}(w).
  $$
  Conversely, 
  $w'$ belongs to the right-hand side if and only if there is a face
  $w''$ of $\Gamma(\Delta,P)$ with content $c(w) \cup c(w')$ such that 
  $w''$ is a linear extension of $(P+w)|_{c(w) \cup c(w')}$.
  This implies that the two families are identical.
  \end{proof}

  \section{Shellable complexes}

  A pure Boolean cell complex is a Boolean cell complex in which all 
  cells that are maximal with respect to inclusion have the same dimension. 
  We define the class of shellable Boolean cell complexes as follows.
  A finite Boolean cell complex $\Gamma$ is called shellable if $\Gamma$
  satisfies one of the following two conditions:      
  \begin{itemize}
    \item[(i)]
      $\Gamma = 2^\Omega$ for some finite set $\Omega$.  
    \item[(ii)]
      $\Gamma$ is pure and there is a 
      cell $\sigma$ in $\Gamma$ that is contained in 
      a unique maximal cell $\tau$ such that
      $\fdel_\Gamma(\sigma)$ is shellable.
  \end{itemize}
  Note that we allow $\Omega = \emptyset$ and $\Gamma = \{\emptyset\}$
  in (i) and $\sigma = \tau$ in (ii). Also in the situation of (ii)
  we have $\fdel_\Gamma(\sigma) = \Gamma \setminus [\sigma,\tau]$, where
  $[\sigma,\tau] = \{\rho: \sigma
  \subseteq \rho \subseteq \tau\}$.
  
  A simple inductive argument yields the following.

  \begin{proposition} \label{shellingorder}
    Let $\Gamma$ be a pure $d$-dimensional Boolean cell complex and
    let $\tau_1, \ldots, \tau_r$ be its inclusion maximal cells.
    Then $\Gamma$ is shellable if and only if 
    there is an ordered partition of $\Gamma$
    into intervals $[\sigma_1,\tau_1]$, $\ldots$, $[\sigma_r,\tau_r]$
    such that for $1 \leq k \leq r-1$ the union $\bigcup_{i=1}^k [\sigma_i,\tau_i]$ is 
    a shellable Boolean cell complex of dimension $d$.
    In particular, $\Gamma$ is homotopy equivalent to a wedge
    of $d$-spheres. The number of spheres is given by 
    $\# \{ j ~|~\sigma_j = \tau_j \}$.
  \end{proposition}

  We refer to such an ordered partition as described in Proposition
  \ref{shellingorder} as a \emph{shelling order}.

  Indeed the usual definition of a shellable Boolean cell complex $\Gamma$ of
  dimension $d$ postulates that $\Gamma$ is pure and that there is linear
  order $\tau_1, \ldots, \tau_r$ of its $d$-dimensional cells such that
  the intersection of the complex generated by $\tau_1, \ldots, \tau_i$ and 
  the cell $\tau_{i+1}$ is
  shellable of dimension $d-1$ for all $1 \leq i \leq r-1$. It is
  easy to check that the existence of such an ordering is equivalent
  to the existence of a shelling order in our sense. 
 
  \begin{lemma}
    If $\Gamma$ is a pure Boolean cell complex and $\rho$ is a face such
    that $\Gamma^1 = \fdel_\Gamma(\rho)$ and
    $\Gamma^2 = \Star_\Gamma(\rho)$ are shellable of dimension $d$,
    then $\Gamma$ is shellable of dimension $d$.
    \label{shell-lem}
  \end{lemma}
  \begin{proof}
    Let 
    $$
      [\alpha_{i1},\beta_{i1}], \ldots, [\alpha_{ir_i},\beta_{ir_i}]
    $$
    be a shelling order of $\Gamma^i$ for $i = 1,2$. 
    Note that each $\beta_{2j}$ contains the face $\rho$, because
    each maximal face of a star complex $\Star_\Gamma(\rho)$ contains
    $\rho$. In the interval $[\alpha_{2j},\beta_{2j}]$, 
    there is a unique minimal face $\gamma_{2j}$ containing
    $\rho$. Namely, the intersection of two such faces would again
    lie in the interval and contain $\rho$. We claim that
    the above shelling order on $\Gamma^1$, together with
    $$
      [\gamma_{21},\beta_{21}], \ldots, [\gamma_{2r_2},\beta_{2r_2}],
    $$
    yields a shelling order on $\Gamma = \Gamma^1 \cup \Gamma^2$.
    Namely, the faces in $[\alpha_{2j},\beta_{2j}] \setminus
    [\gamma_{2j},\beta_{2j}]$ are all contained in 
    $\Gamma^1$. In particular, the family obtained by removing 
    $[\gamma_{2j},\beta_{2j}], \ldots, [\gamma_{2r_2},\beta_{2r_2}]$
    is a pure Boolean cell complex for each $j \in [r]$.
    By an induction argument, starting with
    $\Gamma^1$, we hence obtain that $\Gamma^1 \cup \Gamma^2$ is
    shellable.
  \end{proof}

  \begin{proof}[Proof of Theorem {\rm\ref{shell-thm} (i)}]
    If $P$ is a linear order, then $\Gamma := \Gamma(\Delta,P)$ is
    isomorphic to $\Delta$ and hence shellable.
    Otherwise, let $a \in S$ be maximal such that $a$ is incomparable to
    some other element in $S$ (with respect to the order on $P$). Let $b
    \in S \setminus a$ be minimal such that $a$ and $b$ are
    incomparable.

    Note that if $c <_P b$, then $c <_P a$ by minimality of
    $b$. Analogously, if $a <_P c$, then $b <_P c$.
    This implies that if $w$ is a face of $\Gamma$ such that
    $\{a,b\} \not\subseteq c(w)$, then $w$ is a linear extension of 
    $(P + ba)_{c(w)}$.  Namely, by Lemma~\ref{extend-lem}, there is a
    linear extension $w'$ of $P|_{c(w) \cup \{a,b\}}$ such that
    $w$ is a subword of $w'$. Suppose that $a$ appears before $b$ in
    $w'$. Then all elements $c$ between $a$ and $b$ in $w'$ are
    incomparable to $a$ and $b$ with respect to $P$ by the above
    properties. In particular, 
    we may insert $a$ just before $b$ or insert $b$ just after $a$
    and obtain a word with the desired properties.

    As a consequence, we have that 
    $$
      \Gamma(\Delta,P+ba) = \fdel_\Gamma(ab).
    $$
    In particular, if $ab$ is not in $\Delta$, then $\Gamma$
    coincides with $\Gamma(\Delta,P+ba)$. By induction on $P$, we
    obtain that $\Gamma(\Delta,P)$ is shellable in this case.

    If $ab \in \Delta$, then define 
    \begin{eqnarray*}
      \Gamma^1 &=& \Gamma(\Delta,P+ba) = \fdel_\Gamma(ab);\\
      \Gamma^2 &=& \Gamma(\Star_\Delta(ab),P+ab)
      = \Star_\Gamma(ab);
    \end{eqnarray*}
    the very last equality is a consequence of Lemma~\ref{star-lem}.
    Note that $\Gamma = \Gamma^1 \cup \Gamma^2$.
    By induction on $P$, we have that $\Gamma^1$ is shellable. Moreover,
    since
    $$
      \Star_\Delta(ab) = 2^{\{a,b\}} * \link_\Delta(ab)
      \cong \cone^2(\link_\Delta(ab)),
    $$
    $\Star_\Delta(ab)$ is shellable; shellability of simplicial
    complexes is closed under links and cones. Induction on $P$ yields
    that $\Gamma^2$ is shellable.

    Using Lemma~\ref{shell-lem}, we deduce that
    $\Gamma = \Gamma^1 \cup \Gamma^2$ is shellable, which concludes
    the proof.
  \end{proof}

   \begin{example} \label{matroidexample} Let $\Delta = \Delta_M$ be the simplicial complex of
     independent sets of a matroid $M$ of rank $n$. Is is well-known
     that $\Delta_M$ is shellable (see \cite{Bjorner92} for this fact and
     further background on matroids). Thus $\Gamma(\Delta_M)$ is 
     shellable and hence its homology is concentrated in top dimension.
     As a consequence, 
     \begin{eqnarray} \label{newmatroid} 
       \rank_\integers \widetilde{H}_{n-1}(\Gamma(\Delta_M);\integers) & = &
           \sum_{F \in \Delta_M} (-1)^{n-\# F} \#F!.
     \end{eqnarray}
     This clearly is a matroid invariant. If $\Delta_M$ is the full
     simplex, then already Farmer's results \cite{Farmer} show that the
     left-hand side of (\ref{newmatroid}) equals
     the number of fixed point free permutations on $\# M$ letters.
     We have not been able to recognize the numerical value in
     (\ref{newmatroid}) for other matroids $M$.
     Does there exist a `nice' class of
     combinatorial objects counted by this value?
  \end{example}

  A special case of the preceding example also appears in \cite{vanderKallen} and \cite[Lemma 2.1]{Suslin}.
  There the following situation is considered. For finite dimensional vector spaces $V$ and $W$
  let $\Delta_{V,W}$ be the simplicial complex of all collections 
  $\{ (v_0,w_0), \ldots, (v_l,w_l) \}$ of
  pairs $(v_i,w_i) \in V \times W$, $0 \leq i \leq l$, such that
  $v_0, \ldots, v_l$ are linearly independent. One easily checks that
  if $V$ and $W$ are vector spaces over a finite field then $\Delta_{V,W}$ 
  indeed is the set of independent sets of a matroid.
  Now van der Kallen's result \cite{vanderKallen}, which says that the homology of $\Gamma(\Delta_{V,W})$
  is concentrated in top dimension, is a special case of Example \ref{matroidexample}.

  Before we proceed we list a few other appearances of complexes $\Gamma(\Delta)$ in 
  algebraic $K$-theory, even though they do in general not correspond to matroids or
  shellable $\Delta$. 
  In connection with work on Grassmann homology \cite{Gerdes} and in \cite{Suslin}, 
  for a finite dimensional vector space $V$ and 
  a number $s \leq \dim V$ the complex $\Gamma(\Delta_{V,s})$ appears
  for the simplicial complex $\Delta_{V,s}$ of
  all collections $\{ v_0, \ldots, v_l \}$ of vectors from $V$ such that
  any subset of size $\leq s$ is linearly independent.
  Again it is crucial that homology vanishes except for the top degree (see \cite[Lemma 2.2]{Suslin} for
  the case $s = \dim V$). 
  In \cite{Kerz} the same vanishing of the homology of the `classical' complex $\Gamma_n$ is applied.
  Finally, in \cite{MirzaiivanderKallen} several classes of Boolean cell complexes
  are studied. For example, for a given ring $R$ the complex $\Gamma(\Delta_R)$ 
  is studied for the simplicial complex $\Delta_R$ of all subsets $\{ x_0, \ldots, x_r \}$
  of $R$ such that the ideal generated by $x_0, \ldots, x_n$ is $R$.
  Again vanishing of homology in low dimensions is applied in algebraic $K$-theory. 
  All these examples have in common that they emerge in the following way.
  Given a group $G$ one searches for a (chain) complex with free action of $G$. 
  This is achieved by considering the cellular chain complex of $\Gamma(\Delta)$ for a 
  $G$-invariant simplicial complex over a ground set with free $G$-action.

  For the proof of Theorem \ref{shell-thm} (ii) it will turn out to be
  profitable to code classes $[w]$ by acyclic orientations.
  Let $\Delta$ be a simplicial complex on the vertex set $S$ and let $G
  = (S,E)$ be a simple graph on the same vertex set.
  To $w \in \Gamma(\Delta)$ we assign a directed graph 
  $D_w = (c(w), E_w)$ with vertex set $c(w)$ and with a directed edge
  from $a$ to $b$ whenever $\{a,b\} \in E$ and $a$ precedes
  $b$ in the word $w$. 

  \begin{lemma} \label{acyclic-lemma}
    Let $\Delta$ be a simplicial complex on ground set $S$ and 
    let $G = (S,E)$ be a simple graph. 
    Then the following hold:
    \begin{itemize}
      \item[(i)] 
         For $w \in \Gamma(\Delta)$ the directed graph $D_w$ is
         acyclic.
      \item[(ii)] 
         For $[w] \in \Gamma/G(\Delta)$ and $w' \in [w]$ we have
         $D_w = D_{w'}$. In particular, the map $[w] \mapsto D_w$ is
         well defined for $[w] \in \Gamma/G(\Delta)$.
      \item[(iii)] 
         For each face $\sigma \in \Delta$, the map 
         $[w] \mapsto D_{[w]}$ provides a bijection 
         between faces of $\Gamma/G(\Delta)$ with content $\sigma$ and
         acyclic orientations of the induced subgraph $G|_\sigma$.
      \item[(iv)] The map $[w] \mapsto D_w$ is an isomorphism of
         partially ordered sets between $\Gamma/G(\Delta)$ 
         and the set of acyclic orientations of induced subgraphs
         $G|_\sigma$ for $\sigma \in \Delta$ ordered by 
         inclusion of vertex and edge sets.
      \end{itemize}
  \end{lemma}
  \begin{proof}
    \begin{itemize}
       \item[(i)] 
          Since edges in $D_w$ are directed from left to right in 
          $w$, the graph $D_w$ cannot have any directed cycles.
       \item[(ii)] 
          $D_{[w]}$ is well-defined, because if $e = ab \in G$ and $a$ 
          appears before $b$ in some representative $w$, then $a$ appears 
          before $b$ in every representative. This is because any sequence 
          of commutations of neighboring letters transforming $w$ into a 
          word in which $b$ appears before $a$ must contain a step in which 
          $a$ and $b$ are transposed, which is forbidden.
       \item[(iii)] 
          To prove that the map is surjective, simply note that 
          every acyclic digraph $D$ on ground set $\sigma$ admits a linear 
          extension $w$, i.e., a linear ordering of $\sigma$ such that all 
          edges of $D$ go from smaller to larger vertices. Then
          clearly $D = D_{w}$.

          To prove that the map is injective, suppose that $[w]$ and $[w']$
          yield the same acyclic orientation $D_w = D_{w'}$. Let $b$
          be the first element of $w$ and write $w = b \gamma$ and $w'
          = a_1a_2\cdots a_r b \gamma'$, where $\gamma$ and $\gamma'$
          denote words and $a_1, \ldots, a_r$ letters. By
          construction, $D_w$ contains no edge directed to $b$. 
          Since $D_w = D_{w'}$ this implies that $b$ is not adjacent
          to any $a_i$, $1 \leq i \leq r$, in $G$.
          In particular, we may apply a sequence of
          commutations on $w'$ to obtain the word $w'' = b a_1a_2\cdots a_r
          \gamma'$. By a simple induction argument, we may transform 
          $a_1a_2\cdots a_r \gamma'$ into $\gamma$ via a sequence of
          commutations, which yields that $w$ and $w''$, and hence $w$ and
          $w'$, belong to the same commutation class $[w]$.
       \item[(iv)] By (iii), it remains to verify that 
          $[w'] \preceq [w]$ if and only if $D_{w'}$ is the subgraph of 
          $D_w$ induced on $c(w')$. But this fact is immediate from the 
          definition of $D_w$.  
    \end{itemize}
  \end{proof}

  \begin{proof}[Proof of Theorem {\rm \ref{shell-thm} (ii)}]
    By Lemma~\ref{acyclic-lemma}, we may identify a given face $[w]$
    of $\Gamma$ with the acyclic orientation $D_w$ of $G|_{c(w)}$
    induced by $[w]$.

    Fix a linear order on $S$.
    For vertices $i$ and $j$ of a digraph $D$, we write $i
    \overset{D}{\longrightarrow} j$ if there is a directed path from
    $i$ to $j$. By convention,  
    $i \overset{D}{\longrightarrow} i$ for all $i$.

    In the following we set up functions on vertex sets of digraphs and an
    order relation on these functions that will later be the key ingredient in the
    definition of the shelling.
 
    For a digraph $D$ on vertex set $\rho$, define a function
    $\delta_D = \delta : \rho \rightarrow \rho$ by 
    $$
    \delta(i) = \min \{j~|~
    i \overset{D}{\longrightarrow} j\}.
    $$
    Note that $\delta(i) \le i$ and $\delta^2(i) = \delta(i)$ for all
    $i$.
    Since $\delta^2 = \delta$, it
    is clear that $\delta(i) = i$ whenever $\# \delta^{-1}(\{i\}) = 1$.
    For two functions $\delta_1, \delta_2 : \rho \rightarrow \rho$,
    say that $\delta_1 > \delta_2$ if $\delta_1(x)
    \ge \delta_2(x)$ for all $x \in \rho$ with strict inequality for
    some $x$.

   \noindent {\sf Claim 1:} 
      Let $D$ be an acyclic orientation of $G|_{\rho}$ and 
      let $x \in \rho$. If $\delta_{D}^{-1}(\{x\}) = \{x\}$ then the
      restriction of $\delta_{D}$ to $\rho \setminus \{ x\}$ coincides
      with $\delta_{D \setminus \{x\}}$. 

    \noindent $\triangleleft$ {\sf Proof:} 
      Suppose that we have a path from a vertex $y \neq x$
      to $x$. By construction, $\delta_{D}(y) = a$ for some $a < x$.
      Since there is no path from $x$ to $a$ it follows that there is a
      path from $y$ to $a$ in $D \setminus \{x\}$ and hence that
      $\delta_{D \setminus \{x\}}(y) = a$. 
    $\triangleright$
 
    \noindent {\sf Claim 2:} Let $D$ be an acyclic orientation
      of $G|_{\rho}$ and let $x \in S \setminus \rho$. Then there is a
      unique acyclic
      orientation $D'$ of $G|_{\rho+x}$ containing $D$ such that the 
      restriction of $\delta_{D'}$ to $\rho$ coincides with 
      $\delta_{D}$ and such that $\delta_{D'}(x) =
      x$. The digraph $D'$ has the property that $\delta_{D'}>
      \delta_{D''}$ for all other digraphs $D''$ on $\rho+x$ containing
      $D$.

    \noindent $\triangleleft$ {\sf Proof:}
      Let $A$ be the subset of $\rho$ consisting of all
      elements $a$ such that $\delta_{D}(a) < x$. Consider an acyclic
      orientation $D''$ of $G|_{\rho+x}$ containing $D$.
      Let $e = xb$ be an edge in $G|_{\rho+x}$. If $b \in A$ and $e$ is
      directed from $x$ to $b$, then $\delta_{D''}(x) \le
      \delta_{D}(b) < x$. 
      If $b \in \rho \setminus A$ and $e$ is directed from $b$ to
      $x$, then $\delta_{D''}(b) \le x < \delta_{D}(b)$. 
      Thus for the conditions in Claim 2 to hold, we must direct
      $e$ from $b$ to $x$ whenever $b \in A$ and from $x$ to $b$
      whenever $b \in \rho \setminus A$. For the particular face
      $D'$ with this property, one easily checks that the conditions
      are indeed satisfied. 
      The final statement in Claim 2 follows immediately.
      $\triangleright$
   
      To show that $\Gamma/G(\Delta)$ is shellable, it suffices to 
      verify the following claim.

    \noindent {\sf Claim 3:} Let $\tau$ be a maximal face of $\Delta$ 
       and let $\sigma \subseteq
       \tau$. Then the family
       $$
         \Gamma_{\sigma,\tau} := \{ D_w \in \Gamma/G(\Delta)~|~\sigma \subseteq c(w) \subseteq
            \tau\}
       $$
       admits a partition into intervals 
       \begin{equation}
          [D_1|_{\sigma_1},D_1], \ldots, 
          [D_r|_{\sigma_r},D_r]
          \label{dipartition-eq}
       \end{equation}
       such that 
       $D_i|_{\sigma_i}$ is not a subdigraph of $D_j$ unless
       $i \le j$ and such that each $D_i$ is an acyclic orientation of
       $G|_{\tau}$. Here, $D_i|_{\sigma_i}$ is the induced subdigraph
       of $D_i$ on the vertex set $\sigma_i$.

    Before we proceed to the proof of Claim 3 we provide the
    arguments that show the sufficiency of Claim 3 for shellability
    of $\Gamma/G(\Delta)$. 

    From the shellability of $\Delta$ we deduce from Proposition
    \ref{shellingorder} that there is a maximal face $\tau$ and a face
    $\sigma \subseteq \tau$ such that $\Delta \setminus [\sigma,
    \tau]$ is shellable. Since $\Gamma/G(\Delta) \setminus
    \Gamma_{\sigma,\tau} = \Gamma/G(\Delta \setminus [\sigma, \tau])$,
    Claim 3 implies by inductive applications of Proposition
    \ref{shellingorder} that $\Gamma/G(\Delta)$ is shellable.

    \noindent $\triangleleft$ {\sf Proof of Claim 3:}
      Let $D_1, \ldots, D_r$ be the acyclic orientations of
      $G|_{\tau}$ ordered such that $i<j$ whenever 
      $\delta_{D_i} > \delta_{D_j}$.
      For any $D_i$, let $X_i$ be the set of
      elements $x\in \tau \setminus \sigma$ such that
      $\delta_{D_i}^{-1}(\{x\}) = \{x\}$. Define $\sigma_i = \tau
      \setminus X_i$. 

      We claim that the intervals $[D_i|_{\sigma_i},D_i]$ yield the
      desired partition. First, repeated application of (i) yields that
      $\delta_{D_i|_{\sigma_i}}$ is the restriction of $\delta_{D_i}$ to
      $\sigma_i$. Moreover, repeated application of (ii) yields that
      $\delta_{D_i}(x) \ge \delta_{D_j}(x)$ for all $x \in \tau$
      whenever $D_i|_{\sigma_i}$ is a subdigraph of $D_j$ and that the
      inequality is strict for some $x$, and hence $i < j$, if $D_i \neq
      D_j$. 
      By a similar argument, one obtains that any digraph
      $D_i|_{\rho}$ such that $\sigma_i \subseteq \rho \subseteq \tau$
      has the same property. 
      In particular, we obtain the desired claim.

      Now, let $[\sigma,\tau]$ be the last interval in the shelling
      order of $\Delta$. By induction, we know that 
      $\Gamma_0 = \Gamma/G(\Delta \setminus [\sigma,\tau])$ is
      shellable. Suppose that we have an ordered partition of the form
      (\ref{dipartition-eq})
      of the 
      remaining family $\{ D \in \Gamma/G(\Delta) \setminus \sigma \subseteq
      V(D) \subseteq \tau\}$ with properties as above.
      For $1 \le i \le r$, define $\Gamma_i = \Gamma_{i-1} \cup 
      [D_i|_{\sigma_i},D_i]$. 

      We claim that each $\Gamma_i$ defines a
      pure Boolean cell complex; by definition and induction on $i$, this
      will imply that each $\Gamma_i$ is a shellable Boolean cell complex. By
      assumption the claim is true for $i=0$. Assume that $i > 0$. All
      maximal cells of  
      $\Gamma_i$ have the same dimension $\dim \Delta$, because this is
      true in $\Gamma_0$, and $\Gamma_i$ is the union of $\Gamma_0$ and
      a sequence of intervals in which each top element has dimension
      $\dim \Delta$. It remains to prove that all subfaces of $D_i$
      belong to $\Gamma_i$. Let $D_i|_{\rho}$ be such a subface. If
      $D_i|_{\rho}$ belongs to $\Gamma_0$, then we are done. Otherwise, 
      $D_i|_{\rho} \in [D_j|_{\sigma_j}, D_j]$ for some $1 \leq j \leq r$.
      By construction, $j \le i$, which implies that 
      $D_i|_{\rho} \in \Gamma_i$ as desired.
    $\triangleright$
  \end{proof}

  \begin{example}
    \label{shell-ex}
    Let $\Delta$ be the simplicial complex on ground set
    $S = \{ 1,\cdots, 5\}$ with maximal faces $1234,
    2346, 3456$. The ordered partition 
    $$
      \mbox{}[\emptyset,1234], [6,2346],[5,2345]
    $$
    defines a shelling order of $\Delta$.
    Let $G$ be the graph with vertex set $S$ and edge set
    $\{12,13,24,34,46,56\}$. 
    Table~\ref{shell-table} provides a shelling order of
    $\Gamma/G(\Delta)$ constructed as in the proof of
    Theorem~\ref{shell-thm} (ii) from the given shelling order
    on $\Delta$ with the natural order on $S$. In the table,
    each acyclic orientation $D_i$ is represented by its
    lexicographically smallest representative. The
    function $\delta_{D_i}$ is represented as a word
    $a_1a_2a_3a_4a_5a_6$, where $a_i = \delta_{D_i}(i)$ if $i \in
    V(D_i)$ and $a_i = *$ otherwise. Underlined values $k$ have the
    property that $\delta^{-1}(\{k\}) = \{k\}$.

    \begin{table}
      \label{shell-table}
      \caption{Shelling order for the complex $\Gamma/G(\Delta)$,
	where
	$\Delta$ and $G$ are defined in Example~\ref{shell-ex}.}
      \noindent
      \begin{tabular}{||cccl||cccl||}
	\hline
	$i$ & $D_i$ & $\delta_{D_i}$ & $D_i[\sigma_i]$ &
	$i$ & $D_i$ & $\delta_{D_i}$ & $D_i[\sigma_i]$ \\
	\hline
	\hline
	1& 1234 & {\underline1}{\underline2}{\underline3}{\underline4}{*}{*} &  $\emptyset$  &
	16 & 2364 & {*}\underline2\underline34{*}4 & 64 \\
	2 & 1243 & {\underline1}{\underline2}33{*}{*} & 43 &
	17 & 2436 & {*}\underline233{*}\underline6 & 436 \\
	3 & 1423 & {\underline1}2{\underline3}2{*}{*} & 42  &
	18 & 2643 & {*}\underline233{*}3 & 643 \\
	4 & 1342 & {\underline1}222{*}{*} & 342 &
	19 & 4236 & {*}2\underline32{*}\underline6 & 426 \\
	5 & 3124 & 1{\underline2}1{\underline4}{*}{*} & 31   &
	20 & 6423 & {*}2\underline32{*}2 & 642 \\
	6 & 3142 & 1212{*}{*} & 3142 &
	21 & 3426 & {*}222{*}\underline6 & 3426 \\
	7 & 4312 & 1{\underline2}11{*}{*} & 431  &
	22 & 3642 & {*}222{*}2 & 3642 \\
	\cline{5-8}
	8 & 2134 & 11{\underline3}{\underline4}{*}{*} & 21   &
	23 & 3456 & {*}{*}\underline3\underline4\underline5\underline6 & 5 \\
	9 & 2143 & 1133{*}{*} & 2143 &
	24 & 3465 & {*}{*}\underline3\underline455 & 65 \\
	10 & 4213 & 11{\underline3}1{*}{*} & 421  &
	25 & 3645 & {*}{*}\underline34\underline54 & 645 \\
	11 & 2314 & 111{\underline4}{*}{*} & 231  & 
	26 & 3564 & {*}{*}\underline3444 & 564 \\
	12 & 2431 & 1111{*}{*} & 2431 &
	27 & 4356 & {*}{*}33\underline5\underline6 & 435 \\
	13 & 3421 & 1111{*}{*} & 3421 &
	28 & 4365 & {*}{*}3355 & 4365 \\
	14 & 4231 & 1111{*}{*} & 4231 &
	29 & 6435 & {*}{*}33\underline53 & 6435 \\
	\cline{1-4}
	15 & 2346 & {*}\underline2\underline3\underline4{*}\underline6 & 6 &
	30 & 5643 & {*}{*}3333 & 5643 \\
	\hline
      \end{tabular}
    \end{table}
    \
  \end{example}

  An analysis of the proof of Theorem \ref{shell-thm} (ii)
  allows us to describe the rank of the homology groups
  of $\Gamma/G(\Delta)$ for
  shellable $\Delta$. 

  \begin{corollary} \label{homology}
    Let $\Delta$ be a shellable $d$-dimensional
    simplicial complex on ground set $S$ 
    with shelling order $[\sigma_1,\tau_1]$, $\ldots$,
    $[\sigma_r,\tau_r]$. Fix a linear
    order $<$ on $S$. Then the rank of the unique non-vanishing
    reduced homology group $\widetilde{H}_{d} (\Gamma/G(\Delta);\integers)$
    of $\Gamma/G(\Delta)$ equals the number of pairs $(\tau_i,D)$
    where $1 \le i \le r$ and $D$ is an acyclic
    orientation of $G|_{\tau_i}$ such that for all $x \in
    \tau_i\setminus \sigma_i$ there is a $y \in \tau_i \setminus
    \{x\}$ such that one of the following conditions holds.
    \begin{itemize}
       \item[(C1)] $y < x$ and there 
         is a directed path from $x$ to $y$ in $D$.
       \item[(C2)] $y > x$ and there is a directed path from
         $y$ to $x$ in $D$ and for no $z < x$ there is 
         a directed path from $y$ to $z$ in $D$.
    \end{itemize}
  \end{corollary}
  \begin{proof}
     From the proof of Theorem  \ref{shell-thm} (ii)
     and Proposition \ref{shellingorder} we deduce that $\rank_\integers 
     \widetilde{H}_{d} (\Gamma/G(\Delta);\integers)$ is given by the
     number of pairs $(\tau_i,D)$ where $1 \le i \le r$ and $D$ is an
     acyclic orientation of $G|_{\tau_i}$ 
     such that for all $x \in \tau_i \setminus \sigma_i$ we have
     $\delta_D^{-1}(\{x\}) \neq \{ x\}$. 
     We distinguish two cases:
     \begin{itemize}
     \item[(1)] $\delta_D^{-1}(\{x\}) = \emptyset$. In this case
       there is a $y < x$ for which there is a directed path from
       $x$ to $y$ in $D$.
     \item[(2)] $\# \delta_D^{-1}(\{x\}) \geq 2$. In this case
       there is a $y > x$ for which there is a directed path from
       $y$ to $x$ in $D$ and for no $z < y$ there is a directed path
       from $y$ to $z$ in $D$.
     \end{itemize}
     It is easy to see that (1) and (2) are equivalent to (C1)
     and (C2), respectively. 
  \end{proof}

  \begin{example}
    Let $G$ be a graph on the set $S = \{1, \ldots, n\}$ and 
    $\Delta_n = 2^S$ the full simplex. We consider the natural order on $S$. 
    By Corollary \ref{homology} the
    rank of the top homology group of $\Gamma/G(2^S)$ is equal to the
    number of acyclic orientations $D$ of $G$ such that for each vertex 
    $x \in S$ there is a $y \in S$ satisfying at least one of (C1) and (C2) from 
    Corollary \ref{homology}.

    \begin{itemize}
      \item[(i)] 
        For the complete graph $G = K_n$, by the work of Farmer \cite{Farmer} the homology
        rank is known to be the number of fix-point free permutations. 
        It is an interesting question whether there exists a simple bijection
        between such permutations and acyclic orientations of $K_n$
        satisfying the conditions of Corollary \ref{homology}.
     \item[(ii)] 
        Now we consider the graph $G$ on vertex set $S$ with 
        all edges present except for $\{1,2\}, \{2,3\},\ldots,
        \{n-1,n\},\{1,n\}$. 
        To avoid trivialities consider only $n \geq 4$. 
        Computer calculations for $4 \leq n \leq 9$ suggest that

        $$\rank_\integers \widetilde{H}_{n-1}(\Gamma/G(\Delta_n),\integers) 
        - (-1)^n  = a_n,$$
	where $a_n$ is the number of ways to arrange $n$ non-attacking
        kings on an $n \times n$ chessboard with two sides identified
        to form a cylinder, with one king in each row and
        one king in each column. 
        Is this indeed true for all $n \geq 4$ ? Is there a nice
        bijective proof? Note that the left-hand side equals the
        absolute value of the unreduced Euler characteristic of
        $\Gamma/G(\Delta_n)$.
	
        We refer the reader to Abramson and Moser \cite{AbramsonMoser} for more
        information on the number $a_n$. 
	For small values of $n$, we have that $a_4 = 0$, $a_5 = 10$,
        $a_6 = 60$, $a_7 = 462$, $a_8 = 3920$, $a_9 = 36954$, and
        $a_{10} = 382740$.
     \end{itemize}
   \end{example}

  \section{Cohen-Macaulay and sequentially Cohen-Macaulay complexes}
  \label{cm-sec}

    For the formulation of the results of this section
    we need to review some facts about (sequential) Cohen-Macaulayness.
    Recall (see e.g. \cite{BjornerWachsWelker07}) that a simplicial complex $\Delta$ 
    is called sequentially homotopy Cohen-Macaulay ({\SHCM} for short)
    if for all $r \geq 0$ and all $\sigma \in \Delta$ the subcomplex 
    $(\link_\Delta (\sigma))^{\langle r \rangle}$ 
    generated by all maximal faces of dimension $\geq r$ in 
    $\link_\Delta (\sigma)$ is $(r-1)$-connected.
    For $\field$ a field or $\field = \integers$ a simplicial complex $\Delta$ 
    is called sequentially Cohen-Macaulay over $\field$ ({\SCMk} for short)
    if for all $r \geq 0$ and all $\sigma \in \Delta$ the subcomplex 
    $(\link_\Delta (\sigma))^{\langle r \rangle}$ 
    generated by all maximal faces of dimension $\geq r$ in 
    $\link_\Delta (\sigma)$ has vanishing reduced simplicial homology in
    dimensions $0$ through $(r-1)$.

    In order to define {\SHCM}, {\SCMk}, {\HCM} and {\CMk} for partially
    ordered sets we need to introduce the order complex.
    For a partially ordered set $Q = (M,\leq_Q)$ on ground set $M$ we 
    denote by $\Delta(Q) = \{ m_0 <_Q \cdots < m_l ~|~m_i \in M, l \geq -1 \}$
    its order complex. If $Q$ is the face poset of a Boolean 
    cell complex $\Gamma$ then $\Delta(Q)$ is the barycentric
    subdivision of $\Gamma$. 
 
    We call a partially ordered set $Q = (M,\leq_Q)$ on ground
    set $M$ {\SHCM} (resp\mbox{.} {\SCMk}, {\HCM}, {\CMk}) if $\Delta(Q)$ is {\SHCM}
    (resp\mbox{.} {\SCMk}, {\HCM}, {\CMk}). 
    In particular, we call a Boolean cell complex $\Gamma$ 
    {\SHCM} (resp\mbox{.} {\SCMk}, {\HCM}, {\CMk}) if its barycentric subdivision
    $\Delta(\Gamma)$ is {\SHCM} (resp\mbox{.} {\SCMk}, {\HCM}, {\CMk}).

    A partially ordered set $Q = (M,\leq_Q)$ is called pure if all
    inclusionwise maximal faces of $\Delta(Q)$ have the same dimension. 
    For $m \in M$, we denote by $Q_{\leq m}$ the subposet of $Q$ on ground  
    set $M_{\leq m} = \{ m' \in M~|~m' \leq_Q m \}$. We call $Q$ semipure if
    the poset $Q_{\leq m}$ is pure for all $m \in M$.
    The rank of an element $m \in M$ is the dimension of the simplicial
    complex $\Delta(Q_{\leq m})$. 
    Note that if $Q$ is a pure partially ordered set then the concepts 
    {\SHCM} and {\HCM} (resp\mbox{.} {\SCMk} and {\CMk}) coincide. 
    It is well-known and easy to prove that shellable Boolean cell complexes 
    are {\HCM} and hence also {\CMk}.
    For simplicial complexes, it is well-known that the properties of
    being {\SHCM}, {\SCMk}, {\HCM}, and {\CMk} are preserved
    under barycentric subdivision. 

    The key ingredients to the proof of Theorem \ref{cm-thm} are 
    Theorem \ref{shell-thm} and the following results from \cite{BjornerWachsWelker05}.

    \begin{proposition}[{\cite[Theorem 5.1]{BjornerWachsWelker05}}]
      Let $R = (N, \leq_R)$ and $Q = (M,\leq_Q)$ 
      be semipure partially ordered sets and let $f : R \rightarrow Q$ be a
      surjective and rank-preserving map of partially ordered sets.
      \begin{itemize}
        \item[(i)] Assume that for all $m \in M$ the fiber
          $\Delta(f^{-1}(Q_{\leq m}))$ 
          is {\HCM}. If $Q$ is {\SHCM}, then so is $R$. 
        \item[(ii)] Let $\field$ be a field or $\field = \integers$ and
          assume that for all $m \in M$ the fiber $\Delta(f^{-1}(Q_{\leq
	  m}))$ is {\SCMk}. If $Q$ is {\SCMk}, then so is $R$. 
      \end{itemize}
      \label{poset-fiber} 
    \end{proposition}

  \begin{proof}[Proof of Theorem \ref{cm-thm}]
    \
    
    \begin{itemize}
    \item[(i)] Consider the map $\phi : \Gamma(\Delta,P) \rightarrow \Delta$
      that sends an injective word $\omega_1 \cdots \omega_r$ in 
      $\Gamma(\Delta,P)$ to 
      $$\phi(\omega_1 \cdots \omega_r) := \{ \omega_1 ,\ldots, \omega_r\} \in 
      \Delta.$$  
      Clearly, $\phi$ is a monotone map if we consider 
      $\Gamma(\Delta,P)$ and $\Delta$ as 
      posets ordered by the subword order and inclusion respectively.
      Surjectivity is obvious as well. Since the rank of a word
      from $\Gamma(\Delta,P)$ is given by one less than the cardinality
      of its content and since the rank of an element of $\Delta$ is
      again one less than its cardinality the map is rank preserving. 
      Now for a simplex $\sigma \in \Delta$ we study the preimage 
      $\phi^{-1}(\Delta_{\leq \sigma})$, which consists of all 
      $\omega_1 \cdots \omega_r \in \Gamma(\Delta,P)$ for which
      $\{ \omega_1 , \ldots , \omega_r \} \subseteq \sigma$. Hence,
      if we again denote by $P|_\sigma$ the restriction of $P$ to $\sigma$
      we can identify $\phi^{-1}(\Delta_{\leq \sigma})$ with the
      complex $\Gamma(2^\sigma,P)$. Since the full simplex $2^\sigma$
      is shellable, $\Gamma(2^\sigma,P)$ is a shellable Boolean
      cell complex by Theorem \ref{shell-thm} (i). 
      Therefore, $\Gamma(2^\sigma,P)$ is {\HCM} ({\CMk}). Thus
      by Proposition \ref{poset-fiber} it follows that 
      $\Gamma(\Delta,P)$ is {\SHCM} (resp\mbox{.} {\SCMk}) if $\Delta$ is.
    \item[(ii)] Consider the map $\phi : \Gamma/G(\Delta) \rightarrow \Delta$
      that sends a class $[\omega_1 \cdots \omega_r]$ in 
      $\Gamma(\Delta,P)$ to 
      $$\phi([\omega_1 \cdots \omega_r]) := \{ \omega_1 ,\ldots,
      \omega_r\} \in  \Delta.$$  
      As in the first case, we arrive at the conclusion
      that $\phi$ is a rank preserving, surjective, and
      monotone map.
      
      The preimage $\phi^{-1}(\Delta_{\leq \sigma})$ 
      of a simplex $\sigma \in \Delta$ consists of all 
      $[\omega_1 \cdots \omega_r] \in \Gamma/G(\Delta)$ for
      which $\{ \omega_1 , \ldots , \omega_r \} \subseteq
      \sigma$. As a consequence, we can identify
      $\phi^{-1}(\Delta_{\leq \sigma})$ with the complex
      $\Gamma/G|_{\sigma}(2^\sigma)$, where $G|_{\sigma}$ is
      the induced subgraph of $G$ on the set $\sigma$. 
      Since $2^\sigma$ is shellable, so is
      $\Gamma/G|_{\sigma}(2^\sigma)$ by Theorem \ref{shell-thm} (ii). 
      Therefore,
      $\Gamma/G|_{\sigma}(2^\sigma)$ is {\HCM} ({\CMk}). Thus
      by Proposition \ref{poset-fiber} it follows that the poset
      $\Gamma/G(\Delta)$ is {\SHCM} (resp\mbox{.}  
      {\SCMk}) if $\Delta$ is.
    \end{itemize}
  \end{proof}

  \section{Complexes of injective words are partitionable}

  A cell complex $\Gamma$ is {\em partitionable} if $\Gamma$ admits
  a partition into pairwise disjoint intervals $[\sigma_i,\tau_i]$ such
  that each $\tau_i$ is maximal in $\Gamma$. 
  Any shellable Boolean cell complex is partitionable, but the converse is
  not true in general. In fact, somewhat surprisingly, 
  for $P$ equal to the antichain  $A = (S,\leq_A)$, all complexes of
  injective words are partitionable:

  \begin{proof}[Proof of Theorem \ref{partition-thm}]  
   Let $V$ be the vertex set of $\Delta$ and define a
  total order $\preceq$ on $V$. With the vertices in each face of
  $\Delta$ arranged in increasing order from left to right, this
  induces a lexicographic order on the faces. Specifically, let
  $\sigma \le \tau$ if and only if
  either $\tau$ is a prefix of $\sigma$ or if $\sigma$ is
  lexicographically smaller than $\tau$. 
  For example, for the complex on the vertex set $[4]$ (naturally
  ordered) with maximal faces $123$ and $234$, we have that
  $$
  123 \prec 12\prec 13\prec 1\prec 234\prec 23\prec 24\prec 2\prec
  34\prec 3\prec 4\prec \emptyset.
  $$

  For a word $w = w_0 \cdots w_r$, recall that $c(w) = \{w_0, \ldots,
  w_r\}$. Write $\Gamma := \Gamma(\Delta,A)$.
  Define a function $f : \Gamma \rightarrow \Delta$ by
  $f(w) = c(w) \cup f_0(w)$, where
  $f_0(w)$ is minimal with respect to $\preceq$ among all faces
  of $\link_\Delta(c(w))$.
  Note that $f(w)$ is necessarily a maximal face of $\Delta$.
  Define another function $g : \Gamma \rightarrow \Gamma$ by
  letting $g(w)$ be the shortest prefix $v$ of $w$ such that 
  $f(v) = f(w)$. Let $\Gamma^v$ be the family of faces $w$ such that
  $g(w) = v$. It is clear that the families $\Gamma^v$ constitute a
  partition of $\Gamma$.

  Now, consider a nonempty family $\Gamma^v$. We claim that
  $$
  \Gamma^v = \{ vw~|~c(w) \subseteq f_0(v)\}.
  $$
  Namely, every member of $\Gamma^v$ certainly belongs to the set in
  the right-hand side. Moreover, if $vw$ belongs to this set, then
  $f_0(v) = c(w) \cup f_0(vw)$. Namely, suppose that some face $\sigma$
  of $\link_\Delta(c(vw))$ is smaller than $f_0(vw)$. Then 
  $c(w) \cup \sigma$ is smaller than $c(w) \cup f_0(vw)$, which is a
  contradiction.

  As a conclusion, we may write 
  $$
  \Gamma^v = v \cdot \Gamma' = \{v\sigma~|~\sigma \in \Gamma'\},
  $$
  where $\Gamma'$ is the complex of
  injective words derived from the full simplex on the vertex set
  $f_0(v)$. Since $\Gamma'$ is shellable, we may partition $\Gamma'$
  into intervals $[u,w]$ such that each top cell $w$ is maximal in
  $\Gamma'$. This induces a partition of $\Gamma^v$ into intervals
  $[vu,vw]$ such that each top cell $vw$ is maximal in $\Gamma^v$ and
  hence in $\Gamma$.
  \end{proof}

  The $h$-polynomial $h(\Gamma;t) := \sum_i h_i t^i$ of a
  Boolean cell complex of dimension $d$ is defined by 
  $$
  \sum_i h_i t^i = \sum_i f_i t^i
  (1-t)^{d+1-i} 
  \Longleftrightarrow 
  \sum_i f_i t^i
   = \sum_i h_i t^i(1+t)^{d+1-i},
  $$
  where $f_i$ is the number of cells of dimension $i-1$ in
  $\Gamma$.

  \begin{corollary}
    Let $\Delta$ be a pure simplicial complex.
    Then all coefficients of the $h$-polynomial of $\Gamma(\Delta)$
    are nonnegative.
  \end{corollary}
  \begin{proof}
    By Theorem~\ref{partition-thm}, we may partition
    $\Gamma(\Delta)$ into a disjoint 
    union of intervals $[\sigma_i,\tau_i]$ such that each $\tau_i$ has
    maximum dimension $d$. It follows that 
    $h(\Gamma(\Delta);t) = \sum t^{\#\sigma_i}$.
  \end{proof}

  A partition into intervals of a Boolean cell complex $\Gamma$
  induces a matching of cells such that the only unmatched cells in
  the complex are the ones that 
  form singleton intervals $[\tau,\tau]$ in the partition. 
  Specifically, consider the graph on the set of faces of $\Gamma$,
  where we have an edge between two faces $\sigma$ and $\tau$ whenever
  $\sigma < \tau$ and there is no face $\gamma$  
  for which $\sigma < \gamma < \tau$. Thus this graph is the graph of
  the Hasse diagram of $\Gamma$. Now, for two faces $\sigma < \tau$ of
  $\Gamma$, each interval $[\sigma,\tau]$ is a Boolean lattice and
  therefore the associated Hasse diagram has a perfect matching if and
  only if $\sigma \neq \tau$. In particular, this shows that on the
  Hasse diagram of a partitionable Boolean cell complex there is a
  matching whose only unmatched faces are the ones corresponding to
  one-element intervals.
  In Discrete Morse theory \cite{Forman}, matchings of the Hasse diagram of
  the face poset of a regular CW-complex are used to determine the
  topological structure of the complex.  
  However, in general, discrete Morse theory \cite{Forman} does not apply
  to the matchings constructed above. Namely, a matching relevant to
  discrete Morse theory has to satisfy the additional assumption that
  if one directs all edges from the matching upward by dimension and
  all other edges downward, then the resulting directed graph must be
  acyclic.

  For example, consider the complex with maximal faces $12$
  and $34$. The induced order of the faces is
  $$
  12 \prec 1 \prec  2 \prec 34  \prec 3 \prec 4 \prec \emptyset, 
  $$
  which yields $\Gamma^\emptyset = \{\emptyset,1,2,12,21\}$,
  $\Gamma^{3} = \{3,34\}$ and $\Gamma^4 = \{4,43\}$. 
  We obtain a partition consisting of the four intervals
  $[\emptyset,12]$, $[21,21]$, $[3,34]$ and $[4,43]$, which yields
  a matching including the pairs $\{3,34\}$ and $\{4,43\}$. This is
  illegal in terms of discrete Morse theory.

  \end{document}